\documentclass[11pt]{amsart}
\usepackage{amsthm,amsmath,amstext,amssymb,amscd,euscript, mathrsfs, dsfont,multicol,,times,enumerate,subfig,sidecap}
\usepackage{amsmath}
\usepackage{amssymb}

\usepackage[colorlinks=true, pdfstartview=FitV, linkcolor=blue, citecolor=red, urlcolor=blue]{hyperref} 

\newcommand{\M}{{\mathcal M}}       %
\newcommand{\R}{{\mathbb R}}       
\newcommand{\DD}{{\mathcal D}}
\newcommand{\CC}{{\mathcal C}}

\newcommand{\HH}{{\mathcal H}}

\newcommand{\RR}{{\mathcal R}}

\newcommand{\EE}{{\mathcal E}}

\newcommand{\diam}{\mathop{\rm diam}}
\newcommand{\dist}{{\rm dist}}

\newcommand{\rf}[1]{{(\ref{#1})}}

\newcommand{\supp}{\operatorname{supp}}

\newcommand{\vphi}{{\varphi}}
\newcommand{\ve}{{\varepsilon}}
\newcommand{\vv}{{\vspace{2mm}}}
\newcommand{\vvv}{\vspace{4mm}}
\newcommand{\wt}[1]{{\widetilde{#1}}}

\newcommand{\pom}{{\partial \Omega}}

\newcommand{\bad}{{\mathsf{Bad}}}
\newcommand{\good}{{\mathsf{Good}}}

\newcommand{\HD}{{\mathsf{HD}}}
\newcommand{\LM}{{\mathsf{LM}}}

\def\Xint#1{\mathchoice
{\XXint\displaystyle\textstyle{#1}}%
{\XXint\textstyle\scriptstyle{#1}}%
{\XXint\scriptstyle\scriptscriptstyle{#1}}%
{\XXint\scriptscriptstyle\scriptscriptstyle{#1}}%
\!\int}
\def\XXint#1#2#3{{\setbox0=\hbox{$#1{#2#3}{\int}$ }
\vcenter{\hbox{$#2#3$ }}\kern-.58\wd0}}

\def\avint{\Xint-}

\newtheorem{theorem}{Theorem}[section]
\newtheorem{lemma}[theorem]{Lemma}

\newtheorem{corollary}[theorem]{Corollary}

\newtheorem*{lemma*}{Lemma}
\newtheorem*{theorem*}{Theorem}

\theoremstyle{definition}

\theoremstyle{remark}
\newtheorem{rem}[theorem]{\bf Remark}

\numberwithin{equation}{section}

\newcommand{\brem}{\begin{rem}}
\newcommand{\erem}{\end{rem}}

\def\d{\partial}

\makeatletter
\def\@tocline#1#2#3#4#5#6#7{\relax
  \ifnum #1>\c@tocdepth 
  \else
    \par \addpenalty\@secpenalty\addvspace{#2}%
    \begingroup \hyphenpenalty\@M
    \@ifempty{#4}{%
      \@tempdima\csname r@tocindent\number#1\endcsname\relax
    }{%
      \@tempdima#4\relax
    }%
    \parindent\z@ \leftskip#3\relax \advance\leftskip\@tempdima\relax
    \rightskip\@pnumwidth plus4em \parfillskip-\@pnumwidth
    #5\leavevmode\hskip-\@tempdima
      \ifcase #1
       \or\or \hskip 1em \or \hskip 2em \else \hskip 3em \fi%
      #6\nobreak\relax
    \dotfill\hbox to\@pnumwidth{\@tocpagenum{#7}}\par
    \nobreak
    \endgroup
  \fi}
\makeatother

\begin{document}

\title{Rectifiability of harmonic measure}

\author[Azzam]{Jonas Azzam}
\address{Jonas Azzam
\\
Departament de Matem\`atiques
\\
Universitat Aut\`onoma de Barcelona
\\
Edifici C Facultat de Ci\`encies
\\
08193 Bellaterra (Barcelona), Catalonia
}
\email{jazzam@mat.uab.cat}

\author[Hofmann]{Steve Hofmann}

\address{Steve Hofmann
\\
Department of Mathematics
\\
University of Missouri
\\
Columbia, MO 65211, USA} \email{hofmanns@missouri.edu}

\author[Martell]{Jos\'e Mar{\'\i}a Martell}

\address{Jos\'e Mar{\'\i}a Martell
\\
Instituto de Ciencias Matem\'aticas CSIC-UAM-UC3M-UCM
\\
Consejo Superior de Investigaciones Cient{\'\i}ficas
\\
C/ Nicol\'as Cabrera, 13-15
\\
E-28049 Madrid, Spain} \email{chema.martell@icmat.es}

\author[Mayboroda]{Svitlana Mayboroda}

\address{Svitlana Mayboroda
\\
Department of Mathematics
\\
University of Minnesota
\\
Minnea\-po\-lis, MN 55455, USA} \email{svitlana@math.umn.edu}

\author[Mourgoglou]{Mihalis Mourgoglou}

\address{Mihalis Mourgoglou
\\
Departament de Matem\`atiques\\
 Universitat Aut\`onoma de Barce\-lo\-na and Centre de Recerca Matem\` atica
\\
Edifici C Facultat de Ci\`encies
\\
08193 Bellaterra (Barcelona), Catalonia
}
\email{mmourgoglou@crm.cat}

\author[Tolsa]{Xavier Tolsa}
\address{Xavier Tolsa
\\
ICREA and Departament de Matem\`atiques
\\
Universitat Aut\`onoma de Bar\-ce\-l\-ona
\\
Edifici C Facultat de Ci\`encies
\\
08193 Bellaterra (Barcelona), Catalonia
}
\email{xtolsa@mat.uab.cat}

\author[Volberg]{Alexander Volberg}
\address{Alexander Volberg
\\
Department of Mathematics
\\
Michigan State University
\\
East Lan\-sing, MI 48824, USA}

\email{volberg@math.msu.edu}

\begin{abstract}
In the present paper we prove that for any open connected set $\Omega\subset\R^{n+1}$, $n\geq 1$, and any $E\subset \partial \Omega$ with $\HH^n(E)<\infty$, absolute continuity of the harmonic measure $\omega$ with respect to the Hausdorff measure on $E$ implies that $\omega|_E$ is rectifiable.
This solves an open problem on harmonic measure which turns out to be an old conjecture even in the planar case $n=1$.
\end{abstract}

\date{\today}

\maketitle

\tableofcontents

\section{Introduction}

Our main result is the following.

\begin{theorem}\label{teo1}
Let $n\geq 1$ and $\Omega\subsetneq\R^{n+1}$ be an open  connected set
and let $\omega:=\omega^p$ be the harmonic measure in $\Omega$ where $p$ is a fixed point in $\Omega$.
Let $E\subset\partial\Omega$ be a subset with Hausdorff measure $\HH^n(E)<\infty$. Then:
\begin{itemize}
\item[(a)] If $\omega$ is absolutely continuous with respect to $\HH^n$ on $E$, then $\omega|_E$ is $n$-rectifiable, in the sense that  $\omega$-almost all of $E$ can be covered by a countable union of $n$-dimensional (possibly rotated) Lipschitz graphs.

\item[(b)]  If $\HH^n$ is absolutely continuous with respect to $\omega$ on $E$, then $E$ is an $n$-rectifiable set, 
in the sense that  $\HH^n$-almost all of $E$ can be covered by a countable union of $n$-dimensional (possibly rotated) Lipschitz graphs.
\end{itemize}
\end{theorem}

Notice that, in particular, the statement (a) ensures that any set $E\subset\partial\Omega$ with $\HH^n(E)<\infty$ and $\omega(E)>0$ where $\omega$ is absolutely continuous with respect to $\HH^n$ contains an $n$-rectifiable subset $F$ with $\HH^n(F)>0$.

We remark that the preceding theorem solves an open problem on harmonic measure which turns out to be an old conjecture even in the planar case $n=1$. See for example Conjecture
10 in Bishop's paper \cite{Bishop-conjectures}.

The metric properties of harmonic measure have attracted the attention of many mathematicians. Fundamental results of Makarov \cite{Mak1}, \cite{Mak2} establish that if $n+1=2$ then the Hausdorff dimension $\dim_\HH \omega = 1$ if the set $\partial\Omega$ is connected (and $\partial\Omega$ is not a point of course). The topology is
somehow felt by harmonic measure, and for a general domain $\Omega$ on the Riemann sphere whose complement has positive logarithmic capacity  there exists a subset of $E\subset \pom$ which supports harmonic measure in $\Omega$ and has Hausdorff dimension at most 1, by a very subtle result of Jones and Wolff \cite{JW}. In particular, the supercritical regime becomes clear on the plane: if $s \in (1, 2)$, $0 < \HH^s(E) < \infty$, then $\omega$ is always singular with respect to $\HH^s|_E)$. However, in the space ($n+1 > 2$) the picture is murkier: on the one hand, Bourgain \cite{Bo} proved that the dimension of harmonic measure always drops: $\dim_\HH \omega < n+1$. On the other hand, even for connected $E=\partial\Omega$, it turns out that $\dim_\HH \omega$ can be strictly bigger than $n$, by a celebrated result of Wolff \cite{W}.

Despite the wide variance in dimension, there are still unique phenomena that link harmonic measure and the geometry of the boundary which only occur in codimension one. In particular, there are deep connections between the absolute continuity of harmonic measure in $\mathbb{R}^{n+1}$ with respect to $n$-dimensional Hausdorff measure and the rectifiability of the underlying set, which has been a subject of thorough investigation for even longer than the aforementioned results:  In 1916 F. and M. Riesz proved that for a simply connected domain in the complex plane, with a rectifiable boundary, harmonic measure is absolutely continuous with respect to arclength measure on
the boundary \cite{RR}. More generally, if only a portion of the boundary is
rectifiable, Bishop and Jones \cite{BJ} have shown that harmonic measure is absolutely
continuous with respect to arclength on that portion. They have also proved that the result of \cite{RR} may fail in the absence of some topological hypothesis (e.g., simple connectedness).

The higher dimensional analogues of \cite{BJ} include absolute continuity of harmonic measure with respect to the Hausdorff measure for Lipschitz graph domains \cite{Da} and non-tangentially accessible (NTA) domains \cite{DJ}, \cite{Se}. To be precise, \cite{Da}, \cite{DJ}, \cite{Se} establish a quantitative scale-invariant result, the $A_\infty$ property of harmonic measure, which in the planar case was proved by Lavrent'ev \cite{Lv}. See also \cite{HM12}, \cite{BL}, \cite{Ba}, \cite{Az}, \cite{AMT0}, \cite{Mo}, \cite{ABHM15} along with \cite{AHMNT} in this context. We shall not give a precise definition of NTA domains here (see \cite{JK}), but let us mention that they necessarily satisfy interior and exterior cork-screw conditions as well as a Harnack chain condition, 
that is, certain quantitative analogues of connectivity and openness, respectively. 
On the other hand,
some counterexamples show that some topological restrictions, even
stronger than in the planar case,
 are needed for the absolute continuity of $\omega$ with respect to $\HH^n$ \cite{Wu}, \cite{Z}.

In the present paper we attack the converse direction. We establish that rectifiability is {\it necessary} for absolute continuity of the harmonic measure. This is a free boundary problem. However, the departing assumption, absolute continuity of the harmonic measure with respect to the Hausdorff measure of the set, is essentially the weakest meaningfully possible from a PDE point of view, putting it completely out of the realm of more traditional work, e.g.,  that related to minimization of functionals. At the same time, absence of any a priori topological restrictions on the domain (porosity, flatness, suitable forms of connectivity) notoriously prevents from using the conventional PDE toolbox. 

To this end, recall that the series of paper by Kenig and Toro \cite{KT1, KT2, KT3} established that in the presence of a Reifenberg flatness condition and Ahlfors-David regularity, $\log k \in VMO$ if and only if $\nu\in VMO$, where $k$ is the Poisson kernel with pole at some fixed point, and $\nu$ is the unit normal to the boundary. Moreover, given the same background hypotheses, the condition that $\nu \in VMO$  is
equivalent to a uniform rectifiability condition with vanishing trace.  
 Going further, in \cite{HMU} the authors proved that, in a uniform domain $\Omega$ having an Ahlfors-David regular boundary, scale-invariant bounds on $k$ in $L^p$ (or the weak-$A_\infty$ condition for harmonic measure) imply uniform rectifiability of $\partial\Omega$. Moreover, 
the same result was recently established, more generally, for an open set $\Omega$ 
(not necessarily connected),
satisfying an interior 
cork-screw (i.e., interior porosity) condition and having an Ahlfors-David regular boundary,  see \cite{HMIV}. An important underlying thread of all these results is a priori topological restrictions on the domain, for instance, Ahlfors-David regularity immediately implies porosity: there is $r_0>0$ so that every ball $B$ centered at $\pom$ of radius at most $r_0$ contains another ball $B'\subset \R^{n+1}\setminus\pom$ with $r(B)\approx r(B')$, with the implicit constant depending only on the Ahlfors-David regularity condition.

The present work develops some intricate estimates on the harmonic measure and the Green function which take advantage of the absolute continuity of harmonic measure in a hostile geometrical environment at hand and ultimately yield bounds on the (suitably interpreted) Riesz transform
$$\RR \mu (x) = \int \frac{x-y}{|x-y|^{n+1}}\,d \mu(y),$$
applied to the harmonic measure $\mu =\omega^p$.
The latter allows us to invoke the recent resolution of the David-Semmes conjecture in co-dimension 1 (\cite{NToV}, \cite{NToV-pubmat}, see also \cite{HMM} in the context of uniform domains), establishing that boundedness of the Riesz transforms implies rectifiability. 
We note that in Theorem~\ref{teo1} connectivity is just a cosmetic assumption needed to make sense of harmonic measure at a given pole. In the presence of multiple components, one can work with one component at a time.

An interesting point is that the original problem on geometry of harmonic measure never mentions any singular integrals whatsoever, let alone non-homogeneous ones. However, it turns out that non-homogeneous harmonic analysis point of view again proved itself very useful, exactly as it happened in solving David-Semmes conjecture in co-dimension 1.

Our paper arises from the union of two separate works, \cite{AMT} and \cite{HMMTV}, which will not be published. In \cite{AMT}, a version of Theorem \ref{teo1} was proved under the additional assumption that the boundary of $\Omega$ is {\it porous} in $E$ (in the sense described above). In \cite{HMMTV} the porosity assumption was removed. Both works,  \cite{AMT} and \cite{HMMTV} are available only as preprints on ArXiv. For the purposes of publication we combined them into the present manuscript.

The organization of the paper is as follows. After recalling some notation in Section 2, we recall and develop some lemmas concerning harmonic measure in Section 3, as well as review the definition of the cubes of David and Mattila \cite{David-Mattila}. One particularly useful result from here which may be of independent interest is Lemma \ref{l:w>G}, an inequality relating harmonic measure and the Green function previously known for NTA domains or for open sets with ADR boundary but which now holds in any bounded domain. The main body of the proof of Theorem \ref{teo1} is contained in Section 4, and in Section 5 we discuss some of its application, such as classifying all sets of absolute continuity for NTA domains.

\subsection*{Acknowledgments} 
\null \hskip-.1cm The second author was supported in part by NSF grant DMS 1361701. The third author has been partially supported by ICMAT Severo Ochoa project SEV-2011-0087 and he acknowledges that the research leading to these results has received funding from the European Research Council under the European Union's Seventh Framework Programme (FP7/2007-2013)/ ERC agreement no. 615112 HAPDEGMT. The fourth author is supported in part by the Alfred P. Sloan Fellowship, the NSF INSPIRE Award DMS 1344235, NSF CAREER Award DMS 1220089 and  NSF UMN MRSEC Seed grant DMR 0212302. The sixth author was supported by the ERC grant 320501 of the European Research Council (FP7/2007-2013) (which also funded the first and fifth authors),  by 2014-SGR-75 (Catalonia), MTM2013-44304-P (Spain), and by the Marie Curie ITN MAnET (FP7-607647). The last author was partially supported by the NSF grant DMS-1265549.
The results of this paper were obtained while the second, third, fourth, sixth and seventh authors were participating in the \textit{Research in Paris} program at the \textit{Institut Henri Poincar\'e}, and while the first author was at the 2015 ICMAT program \textit{Analysis and Geometry in Metric Spaces}. All authors would like to express their gratitude to these institutions for their support and nice working environments.

\vv
\section{Some notation}

We will write $a\lesssim b$ if there is $C>0$ so that $a\leq Cb$ and $a\lesssim_{t} b$ if the constant $C$ depends on the parameter $t$. We write $a\sim b$ to mean $a\lesssim b\lesssim a$ and define $a\sim_{t}b$ similarly.

For sets $A,B\subset \R^{n+1}$, we let 
\[\dist(A,B)=\inf\{|x-y|:x\in A,y\in B\}, \;\; \dist(x,A)=\dist(\{x\},A),\]
We denote the open ball of radius $r$ centered at $x$ by $B(x,r)$. For a ball $B=B(x,r)$ and $\delta>0$ we write $r(B)$ for its radius and $\delta B=B(x,\delta r)$. We let $U_\ve (A)$ to be the $\ve$-neighborhood of a set $A\subset \R^{n+1}$. For $A\subset \R^{n+1}$ and $0<\delta\leq\infty$, we set
\[\HH^{n}_{\delta}(A)=\inf\left\{\textstyle{ \sum_i \diam(A_i)^n: A_i\subset\R^{n+1},\,\diam(A_i)\leq\delta,\,A\subset \bigcup_i A_i}\right\}.\]
Define the {\it $n$-dimensional Hausdorff measure} as
\[\HH^{n}(A)=\lim_{\delta\downarrow 0}\HH^{n}_{\delta}(A)\]
and the {\it $n$-dimensional Hausdorff content} as $\HH^{n}_{\infty}(A)$. We let $m$ denote the Lebesgue measure in $\R^{n}$ so that, for some universal constant $c$ depending on $n$ and all Lebesgue measurable subsets $A\subseteq \R^{n}$, $\HH^{n}(A)=cm(A)$. See Chapter 4 of \cite{Ma} for more details. \\

Given a signed Radon measure $\nu$ in $\R^{n+1}$ we consider the $n$-dimensional Riesz
transform
$$\RR\nu(x) = \int \frac{x-y}{|x-y|^{n+1}}\,d\nu(y),$$
whenever the integral makes sense. For $\ve>0$, its $\ve$-truncated version is given by 
$$\RR_\ve \nu(x) = \int_{|x-y|>\ve} \frac{x-y}{|x-y|^{n+1}}\,d\nu(y).$$
For $\delta\geq0$
 we set
$$\RR_{*,\delta} \nu(x)= \sup_{\ve>\delta} |\RR_\ve \nu(x)|.$$
We also consider the maximal operator
$$\M^n_\delta\nu(x) = \sup_{r>\delta}\frac{|\nu|(B(x,r))}{r^n},$$
In the case $\delta=0$ we write $\RR_{*} \nu(x):= \RR_{*,0} \nu(x)$ and $\M^n\nu(x):=\M^n_0\nu(x)$.\\

In what follows, $\Omega$ will always denote a connected domain. If $\Omega$ is bounded, for $f$ a continuous function on $\d\Omega$, the set of {\it upper functions} for $f$ are the superharmonic functions $h$ on $\Omega$ for which $\liminf_{\Omega\ni x\rightarrow \xi} h(x)\geq f(\xi)$, $\xi\in \pom$. If we define $u_{f}(x)=\inf\{h(x):h\mbox{ is an upper function for }f\}$ for $x\in \Omega$, the usual Perron method shows that $u_{f}$ is a harmonic function.  One can alternatively work with {\it lower functions} by replacing superharmonic by subharmonic, $\liminf$ by $\limsup$ and $\inf$ by $\sup$. By the Riesz representation theorem, there is a Radon measure measure $\omega^{x}$ on $\d\Omega$ satisfying
\[ u_{f}(x)=\int f\, d\omega^{x} \mbox{ for all }f\in C_{c}(\d\Omega),\]
which we call the harmonic measure for $\Omega$ with pole at $x$. We refer the reader to \cite[Chapter 3]{Hel} for full details. 
For unbounded domains, harmonic measure can similarly be defined (see for example \cite[Section 3]{HM12}).

\vv

\section{Preliminaries for harmonic measure}

\subsection{Equivalence between (a) and (b) in Theorem \ref{teo1}}

An easy application of the Radon-Nykodim Theorem shows that the statements (a) and (b) in Theorem (a) are equivalent.
Indeed, assume that the statement (a) holds. To prove (b), consider $E\subset \d\Omega$ such that $\HH^{n}(E)<\infty$ and $\HH^{n}\ll \omega$ on $E$. By the Radon-Nykodim Theorem, there is a function $g \in L^1(\omega^p)$ such that
$\HH^n|_E = g \omega^p|_E$.
Let $F\subseteq E$ be the set where $g >0$. Then $\HH^n|_F = \HH^n|_E$ and $\HH^n\ll \omega^{p}\ll \HH^{n}$ on $F$. Hence $\omega^{p}|_F$ is $n$-rectifiable by the statement (a). So $\HH^n|_F$ is $n$-rectifiable,
or equivalently, $F$ is $n$-rectifiable.

The converse implication (b) $\Rightarrow$ (a) is analogous and is left for the reader.

\subsection{Properties of the Green function}
Let $\EE$ denote the fundamental solution for the Laplace equation in $\R^{n+1}$, so that $\mathcal{E}(x)=c_n\,|x|^{1-n}$ for $n\geq 2$, and 
$\EE(x)=-c_1\,\log|x|$ for $n=1$, $c_1, c_n>0$. A {\it Green function} $G_{\Omega}:\Omega\times \Omega\rightarrow[0,\infty]$ for an open set $\Omega\subseteq \R^{n+1}$ is a function with the following properties: for each $x\in \Omega$, $G_{\Omega}(x,y)=\EE(x-y)+h_{x}(y)$ where $h_{x}$ is harmonic on $\Omega$, and whenever $v_{x}$ is a nonnegative superharmonic function that is the sum of $\EE(x-\cdot)$ and another superharmonic function, then  $v_{x}\geq G_{\Omega}(x,\cdot)$ (\cite[Definition 4.2.3]{Hel}). 

An open subset of $\R^{n+1}$ having a Green function will be called a Greenian set. The class of domains considered in Theorem \ref{teo1} are always Greenian. Indeed, all open subsets of $\R^{n+1}$ are Greenian for $n\geq 2$ (\cite[Theorem 4.2.10]{Hel}); in the plane, if $\HH^{1}(\d\Omega)>0$, then $\d\Omega$ is nonpolar (p. 207 Theorem 11.14 of \cite{HKM}) and domains with nonpolar boundaries are Greenian by Myrberg's Theorem (see Theorem 5.3.8 on p. 133 of \cite{AG}).

For a bounded open set, we may write the Green function exactly \cite[Lemma 5.5.1]{Hel}: for $x,y\in\Omega$, $x\neq y$, define
\begin{equation}\label{green}
G(x,y) = \mathcal{E}(x-y) - \int_{\partial\Omega} \mathcal{E}(x-z)\,d\omega^y(z).
\end{equation}

For  $x\in\R^{n+1}\setminus \Omega$ and $y\in\Omega$, we will also set 
\begin{equation}\label{green2}
G(x,y)=0.
\end{equation}

Note that the equation \rf{green} is still valid for $x\in\R^{n+1}\setminus \overline\Omega$ and $y\in\Omega$ by  \cite[Theorem 3.6.10]{Hel}. 
The case when $x\in\partial\Omega$ and $y\in\Omega$ is more delicate and the identity  \rf{green} may fail.
However, we have the following partial result:

\begin{lemma}\label{lemgreen*}
Let $\Omega$ be a Greenian domain and let $y\in\Omega$. For $m$-almost all $x\in\Omega^c$ we have
\begin{equation}\label{eqdf12}
\mathcal{E}(x-y) - \int_{\partial\Omega} \mathcal{E}(x-z)\,d\omega^y(z)=0.
\end{equation}
\end{lemma}

Clearly, in the particular case where $m(\partial
\Omega)=0$, this result is a consequence of the aforementioned fact that \rf{green} also holds for all 
$x\in\R^{n+1}\setminus \overline\Omega$, $y\in\Omega$, with $G(x,y)=0$. However, Theorem \ref{teo1} deals with arbitrary domains in $\R^{n+1}$ and so we cannot assume
that $m(\partial\Omega)=0$.

\begin{proof}
Let $A\subset\Omega^c$ be a compact set with $m(A)>0$. Observe that the function 
$U_A = \EE * \chi_A$ is continuous, bounded in $\R^{n+1}$, and harmonic in $A^c$. Then, by Fubini we have  for all $\,y\in\Omega$,
\begin{align*}
\int_A\biggl(\mathcal{E}(x-y) - \int_{\partial\Omega} \mathcal{E}(x-z)&\,d\omega^y(z)\biggr)\,dm(x) \\
&=
U_A(y) - \int_{\partial\Omega} \int_E \mathcal{E}(x-z)\,dm(x)\,d\omega^y(z)\\
&= U_A(y) - \int_{\partial\Omega} U_A(z)\,d\omega^y(z) = 0, 
\end{align*}
using that $U_A$ is harmonic in $\Omega\subset A^c$  and bounded on $\partial\Omega$ for the last identity.
Since the compact set $A\subset \Omega^c$ is arbitrary, the lemma follows.
\end{proof}

Let us remark that a slightly more elaborated argument shows that the identity \rf{eqdf12} holds for $\HH^s$-almost
all $x\in\Omega^c$, for any $s>n-1$.

\begin{rem}\label{remgreen*}
As a corollary of the preceding lemma we deduce that
$$G(x,y) = \mathcal{E}(x-y) - \int_{\partial\Omega} \mathcal{E}(x-z)\,d\omega^y(z)\quad \mbox{ for $m$-a.e. $x\in\R^{n+1}$.}$$
\end{rem}
We will also need the following auxiliary result.

\begin{lemma}\label{l:w>G}
Let $n\ge 2$ and $\Omega\subset\R^{n+1}$ be a bounded open connected set.
Let $B=B(x_0,r)$ be a closed ball with $x_0\in\pom$ and $0<r<\diam(\pom)$. Then, for all $a>0$,
\begin{equation}\label{eq:Green-lowerbound}
 \omega^{x}(aB)\gtrsim \inf_{z\in 2B\cap \Omega} \omega^{z}(aB)\, r^{n-1}\, G(x,y)\quad\mbox{
 for all $x\in \Omega\backslash  2B$ and $y\in B\cap\Omega$,}
 \end{equation}
 with the implicit constant independent of $a$.
\end{lemma}

\begin{proof}
Fix $y\in B\cap\Omega$ and note that for every $x\in\partial (2B)\cap\Omega$ 
we have
\begin{equation}\label{eqop1}
G(x,y)\lesssim \frac1{|x-y|^{n-1}}\leq \frac c{r^{n-1}} \leq \frac{c\,\omega^x(aB)}{r^{n-1}\,\inf_{z\in 2B\cap \Omega} \omega^{z}(aB)}
.
\end{equation}

Let us observe that the two non-negative functions 
$$
u(x)=c^{-1}\,G(x, y)\,r^{n-1}\,\inf_{z\in 2B\cap \Omega} \omega^{z}(aB)\qquad\text{ and }\qquad v(x)=\omega^x(aB)
$$
are  harmonic, hence continuous, in  $\Omega\setminus \overline{B}$. Note that \eqref{eqop1} says that $u\le v$ in $\partial (2B)\cap\Omega$ and hence $\lim_{\Omega\setminus \overline{2B}\ni z\to x} (v-u)(z)=(v-u)(x)\ge 0$ for every $x\in \partial (2B)\cap\Omega$. On the other hand, for a fixed $y\in B\cap \Omega$, one has that $\lim_{\Omega \ni z\to x} G(z,y)=0$ for every $x\in\pom$ with the exception of a polar set (\cite[Theorem 5.5.4]{Hel}). Gathering all these we conclude that $\liminf_{\Omega\setminus \overline{2B}\ni z\to x} (v-u)(z)\ge 0$ for every $x\in\partial(\Omega\setminus \overline{2B})$ with the exception of a polar set. Finally, we clearly have that $v\ge 0$ on $\Omega\setminus \overline{2B}$ and also that
$u(x)\lesssim r^{n-1}\,|x-y|^{1-n}\leq 1$ for every $x\in \Omega\setminus \overline{2B}$ and $y\in B\cap\Omega$. Thus $v-u$ is bounded from below in $\Omega\setminus \overline{2B}$. We can now invoke the maximum principle \cite[Lemma 5.2.21]{Hel} to conclude that $u\le v$ on $\Omega\setminus \overline{2\,B}$ and hence in $\Omega\setminus 2\,B$. 
\end{proof}

%

\vv

\subsection{Bourgain's Lemma}

We recall a result of Bourgain from \cite{Bo}. 

\begin{lemma}
\label{lembourgain}
There is $\delta_{0}>0$ depending only on $n\geq 1$ so that the following holds for $\delta\in (0,\delta_{0})$. Let $\Omega\subsetneq \R^{n+1}$ be a  bounded domain, $n-1<s\le n+1$,  $\xi \in \partial \Omega$, $r>0$, and $B=B(\xi,r)$. Then 
\[ \omega^{x}(B)\gtrsim_{n,s} \frac{\mathcal H_\infty^{s}(\partial\Omega\cap \delta B)}{(\delta r)^{s}}\quad \mbox{  for all }x\in \delta B\cap \Omega .\]
\end{lemma}

\begin{proof}
Without loss of generality, we assume $\xi=0$ and $r=1$.  We denote
$$\rho=\frac{\mathcal H_\infty^{s}(\partial\Omega\cap \delta B)}{\delta^{s}}.$$
Let $\mu$ be a Frostman measure supported in $\delta B\cap \partial\Omega$ so that 
\begin{itemize}\itemsep=0.2cm
\item $\mu(B(x,r))\leq r^{s}$ for all $x\in \R^{n+1}$ and $r>0$, 
\item  $\rho\delta^{s} \geq \mu(\delta B)\geq c\rho\delta^{s}$ where $c=c(n)>0$.
\end{itemize}

First assume $n>1$. Define the function
\[
u(x)=\int\frac1{|x-y|^{n-1}}\,d\mu(y),\]
which is harmonic out of $\supp\mu$ and satisfies the following properties:
\begin{enumerate}[(i)]
\item For $x\in \delta B$,
\[ u(x)\geq 2^{1-n}\delta^{1-n}\mu(\delta B)\geq c2^{1-n}\delta^{s-n+1}\rho.\]

\item For every $x\in \R^{n+1}$,  since $\mu(\delta B)\leq \delta^{s}$ and $s>n-1$,
\begin{align*}
u(x)
& =\int_{\delta B}\int_{0}^{|x-y|^{-1}}(n-1)t^{n-2}dtd\mu(y)
=(n-1)\int_{0}^{\infty} t^{n-2}\mu(B(x,t^{-1}))dt\\
& \leq (n-1)\left(\int_{0}^{\frac{1}{2\delta}}t^{n-2}\mu(\delta B)+\int_{\frac{1}{2\delta}}^{\infty} t^{n-2-s}dt\right)\\
& \leq \delta^{s-n+1}(n-1)\left(\frac{2^{-n+1}}{n-1}+\frac{2^{s-n+1}}{s-n+1}\right)
\leq \delta^{s-n+1}\frac{6(n-1)}{s-n+1}.
\end{align*}

%
%

\item For $x\in B^{c}$,
\[u(x)=\int \frac1{|x-y|^{n-1}}\,d\mu(x)
\leq 2^{n-1}\mu(\delta B)\leq 2^{n-1} \rho delta^{s}.\]

\end{enumerate}

Set 
\[ v(x)=\frac{u(x)-\sup_{  B^{c}}  u}{\sup u}.\]
Then
\begin{enumerate}[(a)]
\item $v$ is harmonic in $(\delta B\cap \partial\Omega)^{c}$,
\item $v\leq 1$,
\item $v\leq 0$ on $B^{c}$,
\item for $x\in \delta B$ and $\delta\leq (2^{1-2n}c)^{\frac{1}{n-1}}$
\[v(x) \geq   \frac{c2^{1-n}\delta^{s-n+1}\rho-2^{n-1}\rho\delta^{s}}{\frac{6(n-1)}{s-n+1}\delta^{s-n+1}}
\geq c\frac{s-n+1}{6(n-1)}2^{-n}\rho.
\]
\end{enumerate}
Let $\phi$ be any continuous compactly supported function such that $\phi\equiv 1$ on $B$ and  $0\le \phi\le 1$. Note that the previous items imply that $v$ is subharmonic in $\Omega$ and satisfies $\limsup_{\Omega \ni x\rightarrow \xi} v(x)\leq \phi(\xi)$, that is, $v$ is a lower function for $\phi$. Hence the the usual Perron method in bounded domains gives that $\int \phi d\omega^{x} \ge v(x)$ for every $x\in \Omega$ (let us recall that $\int \phi d\omega^{x}$ is precisely the sup of all lower functions for $\phi$).
Taking next the infimum over all such $\phi$, we get  $\omega^{x}(\overline{B})\geq v(x)$ for every $x\in \Omega$, and the lemma follows easily. 
\vv

Now we consider the case $n=1$. Define $\rho$ and $\mu$ just as before but now set 
\[
u(x)=\int \log\frac{1}{|x-y|}d\mu(y).\]
Again, this is harmonic off of $\supp \mu$ and satisfies the following properties:
\begin{enumerate}[(i)]
\item For $x\in \delta B$,
\[u(x)\geq \log \frac{1}{2\delta}\mu(\delta B)\geq c\rho\delta^{s}\log\frac{1}{2\delta}.\]

\item Also, for $x\in B$, since $s\in (0,2)$ (assuming $\delta<\frac14$) 
\begin{align*}
u(x)
& 
\le
\int_{\delta B}\int_{\frac12}^{|x-y|^{-1}}\frac{1}{t}dtd\mu(y)
=
\int_{\frac12}^{\infty} \frac{\mu(B(x,t^{-1}))}{t}dt
\leq \int_{\frac12}^{\frac{1}{2\delta}}\frac{\mu(\delta B)}{t}dt+\int_{\frac{1}{2\delta}}^{\infty} \frac{1}{t^{1+s}}dt\\
& 
=
\delta^{s} \log\frac{1}{\delta}+\frac{2^{s}\delta^{s}}{s}
\leq \frac{8}{s}\delta^{s}\log\frac{1}{2\delta}
\end{align*}
\item For $x\in B^{c}$ (assuming $\delta<\frac{1}{4}$)
\[
u(x)\le \mu(\delta B)\log 2 \leq  \min\{\rho,1\} \delta^{s} \log 2\leq  \frac{2}{s}\min\{\rho,1\}\delta^{s}\log\frac{1}{2\delta}.\]

\item The previous two estimates give $u\leq \frac{8}{s} \delta^{s}\log\frac{1}{2\delta} $.\end{enumerate}
Define $v$ just as before. For $x\in \delta B$ and $\delta<e^{-\frac{2\log 2}{c}}/2$
\[
v(x)
\geq 
\frac{c\rho \delta^{s}\log\frac{1}{2\delta}-\rho\delta^{s}\log 2}{  \frac{6}{s} \delta^{s}\log\frac{1}{2\delta}}
\geq \frac{cs}{16}\rho.
\]
From here the proof continues just as before and this completes the case $n=1$.
\end{proof}
\vv

\subsection{The dyadic lattice of David and Mattila}\label{secdya}

Now we will consider the dyadic lattice of cubes
with small boundaries of David-Mattila associated with $\omega^p$. This lattice has been constructed in \cite[Theorem 3.2]{David-Mattila} (with $\omega^p$ replaced by a general Radon measure). 
Its properties are summarized in the next lemma.

\begin{lemma}[David, Mattila]
\label{lemcubs}
Consider two constants $C_0>1$ and $A_0>5000\,C_0$ and denote $W=\supp\omega^p$. Then there exists a sequence of partitions of $W$ into
Borel subsets $Q$, $Q\in \DD_k$, with the following properties:
\begin{itemize}
\item For each integer $k\geq0$, $W$ is the disjoint union of the ``cubes'' $Q$, $Q\in\DD_k$, and
if $k<l$, $Q\in\DD_l$, and $R\in\DD_k$, then either $Q\cap R=\varnothing$ or else $Q\subset R$.
\vv

\item The general position of the cubes $Q$ can be described as follows. For each $k\geq0$ and each cube $Q\in\DD_k$, there is a ball $B(Q)=B(z_Q,r(Q))$ such that
$$z_Q\in W, \qquad A_0^{-k}\leq r(Q)\leq C_0\,A_0^{-k},$$
$$W\cap B(Q)\subset Q\subset W\cap 28\,B(Q)=W \cap B(z_Q,28r(Q)),$$
and
$$\mbox{the balls\, $5B(Q)$, $Q\in\DD_k$, are disjoint.}$$

\vv
\item The cubes $Q\in\DD_k$ have small boundaries. That is, for each $Q\in\DD_k$ and each
integer $l\geq0$, set
$$N_l^{ext}(Q)= \{x\in W\setminus Q:\,\dist(x,Q)< A_0^{-k-l}\},$$
$$N_l^{int}(Q)= \{x\in Q:\,\dist(x,W\setminus Q)< A_0^{-k-l}\},$$
and
$$N_l(Q)= N_l^{ext}(Q) \cup N_l^{int}(Q).$$
Then
\begin{equation}\label{eqsmb2}
\omega^p(N_l(Q))\leq (C^{-1}C_0^{-3d-1}A_0)^{-l}\,\omega^p(90B(Q)).
\end{equation}
\vv

\item Denote by $\DD_k^{db}$ the family of cubes $Q\in\DD_k$ for which
\begin{equation}\label{eqdob22}
\omega^p(100B(Q))\leq C_0\,\omega^p(B(Q)).
\end{equation}
We have that $r(Q)=A_0^{-k}$ when $Q\in\DD_k\setminus \DD_k^{db}$
and
\begin{equation}\label{eqdob23}
\omega^p(100B(Q))\leq C_0^{-l}\,\omega^p(100^{l+1}B(Q))\quad
\end{equation}
for all $l\geq1$ such that $100^l\leq C_0$ and $Q\in\DD_k\setminus \DD_k^{db}$.
\end{itemize}
\end{lemma}

\vv

We use the notation $\DD=\bigcup_{k\geq0}\DD_k$. Observe that the families $\DD_k$ are only defined for $k\geq0$. So the diameters of the cubes from $\DD$ are uniformly
bounded from above.
Given $Q\in\DD_k$, we denote $J(Q)=k$. 
We set
$\ell(Q)= 56\,C_0\,A_0^{-k}$ and we call it the side length of $Q$. Notice that 
$$\frac1{28}\,C_0^{-1}\ell(Q)\leq \diam(B(Q))\leq\ell(Q).$$
Observe that $r(Q)\sim\diam(B(Q))\sim\ell(Q)$.
Also we call $z_Q$ the center of $Q$, and the cube $Q'\in \DD_{k-1}$ such that $Q'\supset Q$ the parent of $Q$.
 We set
$B_Q=28 \,B(Q)=B(z_Q,28\,r(Q))$, so that 
$$W\cap \tfrac1{28}B_Q\subset Q\subset B_Q.$$

We assume $A_0$ big enough so that the constant $C^{-1}C_0^{-3d-1}A_0$ in 
\rf{eqsmb2} satisfies 
$$C^{-1}C_0^{-3d-1}A_0>A_0^{1/2}>10.$$
Then we deduce that, for all $0<\lambda\leq1$,
\begin{align*}
\omega^p\bigl(\{x\in Q:\dist(x,W\setminus Q)\leq \lambda\,\ell(Q)\}\bigr) + 
\omega^p\bigl(\bigl\{x\in 3.5B_Q:\dist&(x,Q)\leq \lambda\,\ell(Q)\}\bigr)\\
&\leq
c\,\lambda^{1/2}\,\omega^p(3.5B_Q).
\end{align*}

We denote
$\DD^{db}=\bigcup_{k\geq0}\DD_k^{db}$.
Note that, in particular, from \rf{eqdob22} it follows that
\begin{equation}\label{eqdob*}
\omega^{p}(3B_{Q})\leq \omega^p(100B(Q))\leq C_0\,\omega^p(Q)\qquad\mbox{if $Q\in\DD^{db}.$}
\end{equation}
For this reason we will call the cubes from $\DD^{db}$ doubling. 

As shown in \cite[Lemma 5.28]{David-Mattila}, every cube $R\in\DD$ can be covered $\omega^p$-a.e.\
by a family of doubling cubes:
\vv

\begin{lemma}\label{lemcobdob}
Let $R\in\DD$. Suppose that the constants $A_0$ and $C_0$ in Lemma \ref{lemcubs} are
chosen suitably. Then there exists a family of
doubling cubes $\{Q_i\}_{i\in I}\subset \DD^{db}$, with
$Q_i\subset R$ for all $i$, such that their union covers $\omega^p$-almost all $R$.
\end{lemma}

The following result is proved in \cite[Lemma 5.31]{David-Mattila}.
\vv

\begin{lemma}\label{lemcad22}
Let $R\in\DD$ and let $Q\subset R$ be a cube such that all the intermediate cubes $S$,
$Q\subsetneq S\subsetneq R$ are non-doubling (i.e.\ belong to $\DD\setminus \DD^{db}$).
Then
\begin{equation}\label{eqdk88}
\omega^p(100B(Q))\leq A_0^{-10n(J(Q)-J(R)-1)}\omega^p(100B(R)).
\end{equation}
\end{lemma}


Given a ball $B\subset \R^{n+1}$, we consider its $n$-dimensional density:
$$\Theta_\omega(B)= \frac{\omega^p(B)}{r(B)^n}.$$

From the preceding lemma we deduce:

\vv
\begin{lemma}\label{lemcad23}
Let $Q,R\in\DD$ be as in Lemma \ref{lemcad22}.
Then
$$\Theta_\omega(100B(Q))\leq C_0\,A_0^{-9n(J(Q)-J(R)-1)}\,\Theta_\omega(100B(R))$$
and
$$\sum_{S\in\DD:Q\subset S\subset R}\Theta_\omega(100B(S))\lesssim_{A_{0},C_{0}}\,\Theta_\omega(100B(R)).$$
\end{lemma}

For the easy proof, see
 \cite[Lemma 4.4]{Tolsa-memo}, for example.

\vv
From now on we will assume that $C_0$ and $A_0$ are some big fixed constants so that the
results stated in the lemmas of this section hold. 

\vv

\section{The proof of Theorem \ref{teo1}}

\subsection{The strategy}
It is enough to prove the statement (a) in the theorem.
We fix a point $p\in\Omega$ far from the boundary to be specified later.
To prove that $\omega^p|_E$ is rectifiable we will show that any subset of positive harmonic measure 
of $E$ contains another subset $G$ of positive harmonic measure such that $\RR_*\omega^p(x)<\infty$ in $G$.
Applying a deep theorem essentially due to Nazarov, Treil and Volberg, one deduces that $G$ contains yet another
subset $G_0$ of positive harmonic measure such that $\RR_{\omega^p|_{G_0}}$ is bounded in $L^2(\omega^p|_{G_0})$. Then from
the results of Nazarov, Tolsa and Volberg in \cite {NToV} and \cite{NToV-pubmat}, it follows that $\omega^p|_{G_0}$ is $n$-rectifiable.
This suffices to prove the full $n$-rectifiability of $\omega^p|_E$.

One of the difficulties of Theorem \ref{teo1} is due to the fact that the non-Ahlfors  regularity of $\partial\Omega$ makes it difficult to apply some usual tools from potential of theory, such as the ones developed by Aikawa in \cite{Ai1} and \cite{Ai2}. In our proof we solve this issue by applying some stopping time arguments involving the harmonic measure and a suitable Frostman measure. 

The connection between harmonic measure and the Riesz transform is already used, at least implicitly, in the work of Hofmann, Martell and Uriarte-Tuero \cite{HMU}, and more explicitely in the paper by Hofmann, Martell and Mayboroda \cite{HMM}.
Indeed, in \cite{HMU}, in order to prove the uniform rectifiability of $\partial\Omega$,
the authors rely on the study of a square function related to the double gradient of the single layer potential and the application of an appropriate rectifiability criterion due to David and Semmes \cite{DS}. Note that the gradient of the single layer potential coincides with the Riesz transform away from the boundary.\\

\subsection{Reduction to bounded domains}

We now begin the proof of Theorem \ref{teo1} (a) in earnest, and begin by reducing it to the bounded case.

\begin{lemma}
If Theorem \ref{teo1} (a) holds for $\Omega$ bounded, then it holds for all $\Omega\subseteq \R^{n+1}$, $n\geq 1$.
\end{lemma}

\begin{proof}
Here we will follow the construction of harmonic measure from \cite[Section 3]{HM12}. To that end we fix $x_0\in\pom$ and $N$ large enough and set $\Omega_N=\Omega\cap B(x_0,2N)$, whose harmonic measure is denoted by $\omega_N$. We may assume that $E$ in the hypotheses of Theorem \ref{teo1} is bounded (otherwise we may chopped into bounded pieces and work with each of them separately). We assume that $N$ is large enough (say $N\ge N_0$) so that $E\subset B(x_0,N/2)$ and observe that from \cite[Section 3]{HM12} and the inner and outer regularity of harmonic measure  we can easily see that for every fixed $x\in\Omega$, $\omega_N^x(A)\nearrow \omega^x(A)$as $N\to\infty$  for every Borel bounded set $A\subset \pom$. In particular, $\omega^p|_E\ll \HH^n|_E$ implies that 
$\omega^p_N|_E\ll \HH^n|_E$. Note that in the process we may have lost the connectivity of $\Omega_N$ so we need to take $\Omega_N^p$ the connected component of $\Omega_N$ containing $p$. Call $E_N=\partial\Omega_N^p\cap E$ and note that 
$\omega^p_N|_{E_N}\ll \HH^n|_{E_N}$. If $\HH^n(E_N)=0$ then $\omega^p_N|_{E_N}(E_N)=0$ and hence 
$\omega_N^p|_{E_N}$ is $n$-rectifiable. Otherwise, $0<\HH^n(E_N)\le \HH^n(E)<\infty$ and by the bounded case we conclude that $\omega^p_N|_{E_N}$ is $n$-rectifiable. Note also that $\omega^p_N|_{E\setminus E_N}\equiv 0$ ($\omega^p_N$ is supported on $\partial \Omega_N^p$). Hence, $w_N^p|_{E}$ is $n$-rectifiable, that is, there exists a Borel set $F_N\subset E$ with $\omega^p_N(F_N)=0$ and a countable collection of (possibly rotated) Lipschitz graphs $\{\Gamma_j^N\}_j$ so that $E\subset F_N\cup(\cup_j \Gamma_j^N)$. We next set $F_0=\cap_{N\ge N_0} F_N$ which is a Borel set satisfying that $\omega_N^p(F_0)=0$ for every $N\ge N_0$. This implies that $\omega^p(F_0)=\lim_{N\to\infty} \omega_N^p(F_0)=0$. Clearly, $E\subset F_0\cup(\cup_{N\ge N_0}\cup_j \Gamma_j^N)$ and this shows that $\omega^p|_E$ is rectifiable as desired.
\end{proof}

Thus, assume $\Omega$ is a connected bounded open set and let $E$ be as in Theorem \ref{teo1} (a), fix a point $p\in\Omega$, and consider the harmonic measure $\omega^p$ of $\Omega$ with pole at $p$. 
The reader may think that $p$ is point deep inside $\Omega$.

Let $g\in L^1(\omega^p)$ be such that
$$\omega^p|_E = g\,\HH^n|_{\partial\Omega}.$$
Given $M>0$, let 
$$E_M= \{x\in\partial\Omega:M^{-1}\leq g(x)\leq M\}.$$
Take $M$ big enough so that $\omega^p(E_M)\geq \omega^p(E)/2$, say.
Consider an arbitrary compact set $F_M\subset E_M$ with $\omega^p(F_M)>0$. We will show that there exists $G_0\subset F_M$
with $\omega^p(G_0)>0$ which is $n$-rectifiable. Clearly, this suffices to prove that $\omega^p|_{E_M}$ is $n$-rectifiable,
and letting $M\to\infty$ we get the full $n$-rectifiability of $\omega^p|_E$.

Let $\mu$ be an $n$-dimensional Frostman measure for $F_M$. That is, $\mu$ is a non-zero Radon measure supported on $F_M$
such that 
$$\mu(B(x,r))\leq C\,r^n\qquad \mbox{for all $x\in\R^{n+1}$.}$$
Further, by renormalizing $\mu$, we can assume that $\|\mu\|=1$. Of course the constant $C$ above will depend on 
$\HH^n_\infty(F_M)$, and the same may happen for all the constants $C$ to appear,  but this will not bother us. Notice that $\mu\ll\HH^n|_{F_M}\ll \omega^p$. In fact, for any set $H\subset F_M$,
\begin{equation}\label{Frostman}
\mu(H)\leq C\,\HH^n_\infty(H)\leq C\,\HH^n(H)\leq C\,M\,\omega^p(H).
\end{equation}

\vv

\subsection{The bad cubes}\label{secbad}

Now we need to define a family of bad cubes.
We say that $Q\in\DD$ is bad and we write $Q\in\bad$, if $Q\in\DD$ is a maximal cube satisfying one of the conditions below: 
\begin{itemize}\itemsep=0.2cm
\item[(a)] $\mu(Q)\leq \tau\,\omega^p(Q)$, where $\tau>0$ is a small parameter to be fixed below, or
\item[(b)]  $\omega^p(3B_Q)\geq A\,r(B_Q)^n$, where $A$ is some big constant to be fixed below.
\end{itemize}
The existence maximal cubes is guaranteed by the fact that all the cubes from $\DD$ have side length uniformly bounded from
above (since $\DD_k$ is defined only for $k\geq0$). 
If the condition (a) holds, we write $Q\in\LM$ (little measure $\mu$) and in the case (b), $Q\in\HD$ (high density).
On the other hand, if a cube $Q\in\DD$ is not contained in any cube from $\bad$, we say that $Q$ is good and we write
$Q\in\good$.
 
Notice that 
$$ 
\sum_{Q\in\LM\cap\bad} \mu(Q) \leq \tau \sum_{Q\in\LM\cap\bad} \omega^p(Q) \leq \tau\,\|\omega\|=\tau=\tau\,\mu(F_M).$$
Therefore, taking into account that $\tau\leq1/2$ and that $\omega^p|_{F_M}=g(x)\,\HH^n|_{F_M}$ with $g(x)\geq M$, we have by \eqref{Frostman}
\begin{align*}
\frac12\,\omega^p(F_M)&\leq \frac{1}{2}=\frac12\,\mu(F_M)\leq \mu\Bigl(F_M\setminus \bigcup_{Q\in\LM\cap\bad} Q\Bigr)\\
&\leq C\,\HH^n\Bigl(F_M\setminus \bigcup_{Q\in\LM\cap\bad} Q\Bigr)\leq C\,M\,\omega^p\Bigl(F_M\setminus \bigcup_{Q\in\LM\cap\bad} Q\Bigr)
.
\end{align*}

On the other hand, we claim that $\Theta^{n,*}(x,\omega^p):=\limsup_{r\to 0}\frac{\omega^p(B(x,r))}{(2r)^n}<\infty$ for $\omega^p$-a.e. $x\in E$. Indeed, by \cite[Theorem 2.12]{Ma}, $\lim_{r\rightarrow 0} \omega(B(x,r))/\HH^{n}|_{E}(B(x,r))$ exists and is finite for $\HH^{n}$-a.e.  $x\in E$. 
Also, since we have assumed that $\HH^n(E)<\infty$, \cite[Theorem 6.2]{Ma} gives $\limsup_{r\rightarrow 0} \HH^n|_E(B(x,r))/(2\,r)^n\le 1$ for $\HH^{n}$-a.e.  $x\in E$. Gathering these two we conclude that $\Theta^{n,*}(x,\omega^p)<\infty$ for $\HH^{n}$-a.e.  $x\in E$ and hence for $\omega^p$-a.e. $x\in E$ since by assumption $\omega|_E\ll \HH^{n}|_E$. This shows the claim. 

As a consequence we next claim that for any $\delta>0$ and for $A$ big enough
$$
\omega^p\biggl(F_M\bigcap \bigg(\bigcup_{Q\in\HD}Q\bigg)\biggr) <\delta\,\omega^p(F_M).
$$
Indeed it is straightforward to show that
$$
\omega^p\biggl(F_M\bigcap \bigg(\bigcup_{Q\in\HD}Q\bigg)\biggr)
\le
\omega^p\big(\{x\in E:\ \M^n\omega^p(x)\ge 4^{-n}A\}\big)
$$
and since $\omega^p(F_M)>0$ it suffices to see that $\omega^p(E_k)\to 0$ as $k\to \infty$ with $E_{k}=\{x\in E:\  \M^n\omega^p(x)>k\}$. If this were not true we would have that $\omega^p(E_{k})\geq c>0$ for all $k\ge 1$, and then $\omega^p(\{x\in E:\ \M^n\omega^p(x)=\infty\})\ge c$. We know from the previous claim that $\M^n\omega^p<\infty$,  $\omega^p$-a.e. in $E$,  hence get a contradiction.

From the above estimates it follows that  for $\delta$ small enough 
\begin{equation}\label{eqbig}
\omega^p\biggl(F_M\setminus \bigcup_{Q\in\bad} Q\biggr) >0,
\end{equation}
if $\tau$ and $A$ have been chosen appropriately. 
\vv

For technical reasons we have now to introduce a variant of the family $\DD^{db}$ of doubling cubes
defined in Section \ref{secdya}.
Given some constant $T\geq C_0$ (where $C_0$ is the constant in Lemma \ref{lemcubs}) to be fixed below,
we say that $Q\in\wt\DD^{db}$ if
$$
\omega^p(100B(Q))\leq T\,\omega^p(Q).
$$
We also set $\wt \DD^{db}_k=\wt\DD^{db}\cap \DD_k$ for $k\geq0$.
From \rf{eqdob*} and the fact that $T\geq C_0$, it is clear that $\DD^{db}\subset \wt\DD^{db}$.

\vv

\begin{lemma}\label{lemgg}
 If the constant $T$ is chosen big enough, then 
$$\omega^p\biggl(F_M \cap \bigcup_{Q\in\wt\DD_0^{db}} Q  \setminus
 \bigcup_{Q\in\bad} Q
\biggr) >0.$$
\end{lemma}

Notice that above $\wt\DD_0^{db}$ stands for the family of cubes from the zero level of $\wt\DD^{db}$.

\begin{proof}
By the preceding discussion we already know that 
$$\omega^p\biggl(F_M\setminus \bigcup_{Q\in\bad} Q\biggr) >0.$$
If $Q\not\in\wt\DD^{db}$, then $\omega^p(Q)\leq T^{-1}\omega^p(100B(Q))$. Hence by the finite overlap of the balls
$100B(Q)$  associated with cubes from $\DD_0$ we get
$$\omega^p\biggl(\,\bigcup_{Q\in\DD_0\setminus\wt\DD^{db}} Q\biggr) \leq \frac1T\sum_{Q\in\DD_0}\omega^p(100B(Q))\leq
\frac CT\,\|\omega^p\|=\frac CT.$$
Thus for $T$ big enough we derive
$$\omega^p\biggl(\,\bigcup_{Q\in\DD_0\setminus\wt\DD^{db}} Q\biggr) \leq \frac12\,\omega^p\biggl(F_M\setminus \bigcup_{Q\in\bad} Q\biggr),$$
and then the lemma follows.
\end{proof}
\vv

Notice that for the points $x\in F_M\setminus \bigcup_{Q\in\bad} Q$, from the condition (b) in the definition
of bad cubes, it follows that
$$\omega^p(B(x,r))\lesssim A\,r^n\qquad \mbox{for all $0<r\leq 1$.}$$
Trivially, the same estimate holds for $r\geq1$, since $\|\omega^p\|=1$. So we have
\begin{equation}\label{eqcc0}
\M^n\omega^p(x)\lesssim A\quad \mbox{ for $\omega^p$-a.e.\ $x\in F_M\setminus \bigcup_{Q\in\bad} Q$.}
\end{equation}

\vv

\subsection{The key lemma about the Riesz transform of $\omega^p$ on the good cubes}

\begin{lemma}[Key lemma]
Let $Q\in\good$ be contained in some cube from the family $\wt{\DD}_0^{db}$, and $x\in Q$. Then we have
\begin{equation}\label{eqdk0}
\bigl|\RR_{r(B_Q)}\omega^p(x)\bigr| \leq  C(A,M,T,\tau,d(p)),
\end{equation}
where, to shorten notation, we wrote $d(p)= \dist(p,\partial\Omega)$.
\end{lemma}

\vspace{1mm}
\begin{proof}[\bf Proof of the Key Lemma in the case \boldmath{$n\geq2$}] 
To prove the lemma, clearly we may assume that $r(B_Q)\ll\dist(p,\partial\Omega)$ for any $P\in \good$. 
First we will prove \rf{eqdk0} for $Q\in\wt\DD^{db}\cap\good$.
In this case, by definition we have
\begin{equation}\label{eqcond**}
\mu(Q)> \tau\,\omega^p(Q)\quad \mbox{and}\quad \omega^p(3B_Q)\leq T\,\omega^p(Q).
\end{equation}

Let $\vphi:\R^{n+1}\to[0,1]$ be a radial $\CC^\infty$ function  which vanishes on $B(0,1)$ and equals $1$ on $\R^{n+1}\setminus B(0,2)$,
and for $\ve>0$ and $z\in \R^{n+1}$ denote
$\vphi_\ve(z) = \vphi\left(\frac{z}\ve\right) $ and $\psi_\ve = 1-\vphi_\ve$.
We set
$$\wt\RR_\ve\omega^p(z) =\int K(z-y)\,\vphi_\ve(z-y)\,d\omega^p(y),$$
where $K(\cdot)$ is the kernel of the $n$-dimensional Riesz transform. 

Let $\delta>0$ be the constant appearing in Lemma \ref{lembourgain} about Bourgain's estimate. Consider a ball $\wt B_Q$ centered at some point from $B_Q\cap\partial\Omega$ with $r(\wt B_Q)= \frac{\delta}{10} \,r(B_Q)$ and so that $\mu(\wt B_Q)\gtrsim \mu(B_Q)$, with the implicit constant depending
on $\delta$.
Note that, for every $x,z\in B_Q$, by standard Calder\'on-Zygmund estimates 
$$\bigl|\wt\RR_{r(\wt B_Q)}\omega^p(x)- \RR_{r(B_Q)}\omega^p(z)\bigr|\leq C(\delta)\,\M^n_{r(\wt B_Q)}\omega^p(z),$$
and
$$\M^n_{r(\wt B_Q)}\omega^p(z)\leq C(\delta,A)\qquad \mbox{for all $z\in B_Q$,}$$
 since $Q$ being good implies that
$Q$ and all its ancestors are not from $\HD$.
Thus,
to prove \rf{eqdk0} it suffices to show that
\begin{equation}\label{eqdk00}
\bigl|\wt\RR_{r(\wt B_Q)}\omega^p(x)\bigr| \leq  C(\delta,A,M,T,\tau,d(p)) \quad\mbox{ for the center $x$ of $\wt B_Q$.}
\end{equation}

To shorten notation, in the rest of the proof we will write $r=r(\wt B_Q)$, so that $\wt B_Q=B(x,r)$.
For a fixed $x\in Q\subset\pom$ and $z\in \R^{n+1}\setminus \bigl[\supp(\vphi_r(x-\cdot)\,\omega^p)\cup \{p\}\bigr]$, consider the function
$$u_r(z) = \EE(z-p) - \int \EE(z-y)\,\vphi_r(x-y)\,d\omega^p(y),$$
so that, by Remark \ref{remgreen*},
\begin{equation}\label{eqfj33}
G(z,p) = u_r(z) - \int \EE(z-y)\,\psi_r(x-y)\,d\omega^p(y)\quad \mbox{ for $m$-a.e.   $z\in\R^{n+1}$.}
\end{equation}

Since the kernel of the Riesz transform is
\begin{equation}\label{eqker}
K(x) = c_n\,\nabla \EE(x),
\end{equation}
for a suitable absolute constant $c_n$, we have
$$\nabla u_r(z) = c_n\,K(z-p) - c_n\,\RR(\vphi_{ r}(\cdot-x)\,\omega^p)(z).$$

In the particular case $z=x$ we get
$$\nabla u_r(x) = c_n\,K(x-p) - c_n\,\wt\RR_r\omega^p(x),$$
and thus
\begin{equation}\label{eqcv1}
|\wt\RR_r\omega^p(x)|\lesssim \frac1{d(p)^n} + |\nabla u_r(x)|.
\end{equation}

Since $u_r$ is harmonic in $\R^{n+1}\setminus \bigl[\supp(\vphi_r(x-\cdot)\,\omega^p)\cup \{p\}\bigr]$ (and so in $B(x,r)$),  
we have
\begin{equation}\label{eqcv2}
|\nabla u_r(x)| \lesssim \frac1r\,\avint_{B(x,r)}|u_r(z)|\,dm(z).
\end{equation}
From the identity \rf{eqfj33} we deduce that
\begin{align}\label{eqcv3}
|\nabla u_r(x)| &\lesssim \frac1r\,\avint_{B(x,r)}G(z,p)\,dm(z) + 
\frac1r\,\avint_{B(x,r)}
\int \EE(z-y)\,\psi_r(x-y)\,d\omega^p(y)\,dm(z) \nonumber\\
& =:I + II.
\end{align}
To estimate the term $II$ we use Fubini and the fact that $\supp\psi_r\subset B(x,2r)$:
\begin{align*}
II & \lesssim \frac1{r^{n+2}}\, \int_{y\in B(x,2r)}\int_{z\in B(x,r)} \frac 1{|z-y|^{n-1}} \,dm(z)\,d\omega^p(y)\\
& \lesssim \frac{\omega^p(B(x,2r))}{r^{n}} \lesssim \M^n\omega^p(x).
\end{align*}

We intend to show now that $I\lesssim 1$.
Clearly it is enough to show that
\begin{equation}\label{eqsuf1}
\frac1r\,| G(y,p)|\lesssim 1\qquad\mbox{for all $y\in  B(x,r)\cap\Omega$.}
\end{equation}
To prove this, observe that by Lemma \ref{l:w>G} (with $B= B(x,r)$, $a=2\delta^{-1}$), for all $y\in B(x,r)\cap\Omega$,
we have
$$\omega^{p}(B(x,2\delta^{-1}r))\gtrsim \inf_{z\in B(x,2r)\cap \Omega} \omega^{z}(B(x,2\delta^{-1}r))\, r^{n-1}\,|G(y,p)|.$$
On the other hand, by Lemma \ref{lembourgain}, for any $z\in B(x,2r)\cap\Omega$,
$$\omega^{z}(B(x,2\delta^{-1}r))\gtrsim \frac{\mu(B(x,2r))}{r^n}\geq \frac{\mu(\wt B_Q)}{r^n}.$$
Therefore we have
$$
\omega^{p}(B(x,2\delta^{-1}r))
\gtrsim 
\frac{\mu(\wt B_Q)}{r^n}\, r^{n-1}\,|G(y,p)|,
$$
and thus
$$
\frac1r\,| G(y,p)|\lesssim \frac{\omega^{p}(B(x,2\delta^{-1}r))}{\mu(\wt B_Q)}.
$$
Now, recall that by construction $\mu(\wt B_Q)\gtrsim \mu(B_Q)\geq \mu(Q)$ and
$B(x,2\delta^{-1}r)=2\delta^{-1} \wt B_Q\subset 3B_Q$, since $r(\wt B_Q)=\frac\delta{10}r(B_Q)$, and so
we have
$$
\frac1r\,| G(y,p)|\lesssim 
\frac{\omega^{p}(B(x,2\delta^{-1}r))}{\mu(\wt B_Q)}\lesssim \frac{\omega^{p}(3B_Q)}{\mu(Q)}\lesssim
\frac{\omega^{p}(Q)}{\mu(Q)}\leq C,$$
by \rf{eqcond**}. So \rf{eqsuf1} is proved and the proof of the Key lemma is complete in the case $n\geq2$, $Q\in\wt\DD^{db}$.

\vv

Consider now the case $Q\in\good\setminus\wt\DD^{db}$. Let $Q'\supset Q$ be the cube from $\wt\DD^{db}$ with minimal side length.
The existence of $Q'$  is guaranteed by the fact that we have assumed that $Q$ is contained in some cube from $\wt\DD^{db}_0$.
For all $y\in B_Q$ then we have
$$|\RR_{r(B_Q)}\omega^p (y)| \leq |\RR_{r(B_{Q'})}\omega^p (y)| + C\,\sum_{P\in\DD: Q\subset P\subset Q'}\Theta_\omega(2B_P). 
$$
The first term on the right hand side is bounded by some constant depending on $A,M,\tau,\ldots$
by the previous case since $Q'\in \wt\DD^{db}\cap\good$.
To bound the last sum we can apply Lemma \ref{lemcad23} (because the cubes that are not from $\wt\DD^{db}$ do not belong to $\DD^{db}$ either) and then we get 
$$
\sum_{P\in\DD: Q\subset P\subset Q'}\Theta_\omega(2B_P)
\lesssim 
\Theta_\omega(100\,B(Q'))
\lesssim \Theta_\omega(3B_{Q'}),
$$ 
where we have used  that $Q'\in \wt\DD^{db}$.
Finally, since $Q'\not\in\HD$, we have $\Theta_\omega(3B_{Q'})\lesssim A$. So \rf{eqdk0} also holds in this case.
\end{proof}
\vvv

\begin{proof}[\bf Proof of the Key Lemma in the planar case \boldmath{$n=1$}]

We note  that the arguments to prove Lemma \ref{l:w>G} fail in the planar case. Therefore this cannot be
applied to prove the Key Lemma and some changes are required. 

We follow the same scheme and notation as in the case $n\geq2$ and highlight the important modifications. 
We claim that for any constant $\alpha\in \R$,
\begin{equation}\label{eq1**}
\bigl|\wt\RR_r\omega^p(x)\bigr|\lesssim \frac{1}{r}
\avint_{B(x,r)} |G(y,p) - \alpha|\, 
\,dm(y) + \frac1{d(p)} + \M^1\omega^p(x).
\end{equation}
To check this, we can argue as in the proof of the Key Lemma for $n\geq2$ to get\begin{equation}\label{eq2**}
|\wt\RR_r\omega^p(x)| \lesssim \frac{1}{d(p)} + |\nabla u_r(x)|.
\end{equation}

Since $u_r$ is harmonic in $\R^{2}\setminus \bigl[\supp(\vphi_r(x-\cdot)\,\omega^p)\cup \{p\}\bigr]$ 
for any constant $\alpha'\in\R^n$, we have
\begin{equation}
\label{eq:aasasd}
|\nabla u_r(y)| \lesssim \frac1r\,\avint_{B(y,r)}|u_r(\cdot)-\alpha'|\,dm(z),
\quad
B(y,r)\subset \R^{2}\setminus \bigl[\supp(\vphi_r(x-\cdot)\,\omega^p)\cup \{p\}\bigr].
\end{equation}
Note that this estimate is the same as the one in in \rf{eqcv2} in the case $n\geq2$ with  $\alpha'=0$ and $y=x$. Let $\alpha'=\alpha'(x)=\alpha +\beta\int \psi_{r}(x-z)d\omega^{p}(z)$ where we choose $\beta=\avint_{B(0,3r)}\EE(z)dm(z)$. Recall that $x$ has been fixed and hence $\alpha$ is a constant. We can now apply \eqref{eq:aasasd} with $y=x$ since $B(x,r)\subset \R^{2}\setminus \bigl[\supp(\vphi_r(x-\cdot)\,\omega^p)\cup \{p\}\bigr]$ and identity \rf{eqfj33} to deduce that
\begin{align}\label{eqcv3'}
|\nabla u_r(x)| &\lesssim \frac1r\,\avint_{B(x,r)}|G(z,p)-\alpha|\,dm(z)  \nonumber\\
&\quad + 
\frac1r\,\avint_{B(x,r)}
\int |\EE(z-y)-\beta|\,\psi_r(x-y)\,d\omega^p(y) \,dm(z) \nonumber\\
& =:I + II,
\end{align}
for any $\alpha\in\R$.

To estimate the term $II$ we apply Fubini and that $\EE(\cdot)=-c_1\,\log|\cdot|$ is a BMO function
\begin{align}
II&
\leq  \frac cr\,
\int_{y\in B(x,2r)} \avint_{z\in B(x,r)}\left|\EE(z-y) -\beta\right|\,dm(z)\,d\omega^p(y)
\nonumber\\
&\lesssim
\frac cr\,\int_{y\in B(x,2r)} 
\avint_{B(0,3r)}\left|\EE(z) -\beta\right|\,dm(z)\,d\omega^p(y)
\nonumber\\
&\lesssim
\frac{\omega^p(B(x,2r))}{r}
\nonumber\\
&\leq 
\M^1\omega^p(x).\label{eq3**}
\end{align}
Hence \rf{eq1**} follows from \rf{eq2**}, \rf{eqcv3'}
and \rf{eq3**}

Choosing $\alpha=G(z,p)$ with $z\in B(x,r)$ in \rf{eq1**} and averaging with respect Lebesgue measure for such $z$'s,  we get
$$\bigl|\wt\RR_{r}\omega^p(x)\bigr|\lesssim \frac1{r^5}\iint_{B(x,r)\times B(x,r)} |G(y,p) - G(z,p)|\,dm(y)\,dm(z) + \frac1{d(p)}+ \M^1\omega^p(x),$$
where we understand that $G(z,p)=0$ for $z\not\in\Omega$.
Now for $m$-a.e. $y,z\in B(x,r)$, $p$ far away, and $\phi$ a radial smooth function such that $\phi\equiv 0$ in $B(0,2)$ and $\phi\equiv 1$ in $\R^{2}\setminus B(0,3)$, we write 
\begin{align*}
c_1(G(y,p) - G(z,p)) 
& = 
\log\frac{|z-p|}{|y-p|} - \int_{\pom} \log\frac{|z-\xi|}{|y-\xi|} \,d\omega^p(\xi) \\
& = 
\left(\log\frac{|z-p|}{|y-p|} - \int_{\partial\Omega} \phi\left(\frac{\xi-x}{r}\right)\,\log\frac{|z-\xi|}{|y-\xi|} \,d\omega^p(\xi)\right) \\
&\quad- \int_{\partial\Omega} \left(1-\phi\left(\frac{\xi-x}{r}\right)\right)\log\frac{|z-\xi|}{|y-\xi|} \,d\omega^p(\xi) 
= A_{y,z}  + B_{y,z}.
\end{align*}
Notice that the above identities also hold if $y,z\not\in\Omega$.
Let us observe that  
$$\frac{|z-p|}{|y-p|}\approx 1 \;\; \mbox{ and }\;\; 
\frac{|z-\xi|}{|y-\xi|}\approx 1\quad\mbox{ for $\xi\not \in B(x,2r)$.}
$$
We claim that  
\begin{equation}\label{eqlem33}
|A_{y,z}|\lesssim \frac{\omega^p(B(x,2\delta^{-1}r))}{\inf_{z\in B(x,2r)\cap \Omega} \omega^z(B(x,2\delta^{-1}r))}.
\end{equation}
We defer the details till the end of the proof.
By Lemma \ref{lembourgain}, we get
$$\inf_{z\in B(x,2r)\cap \Omega}\omega^{z}(B(x,2\delta^{-1}r))\gtrsim \frac{\mu(B(x,2r))}{r}\geq \frac{\mu(\wt B_Q)}{r}.$$
and thus
$$\frac{|A_{y,z}|}r\lesssim \frac{\omega^p(B(x,2\delta^{-1}r))}{\mu(\wt B_Q)}\lesssim 
\frac{\omega^{p}(Q)}{\mu(Q)},$$
by the doubling properties of $Q$ (for $\omega^p$) and the choice of $\wt B_Q$.

To deal with the term $B_{y,z}$ we write:
\begin{align*}
|B_{y,z}|&\leq \int_{ B(x,3r)} \left(\left|\log\frac{r}{|y-\xi|}\right| +  \left|\log\frac{r}{|z-\xi|}\right|\right) \,d\omega^p(\xi).
\end{align*}
Thus
\begin{align*}
\iint_{B(x,r)\times B(x,r)} &|B_{y,z}|\,dm(y)\,dm(z) \\
&\lesssim 
r^2\int_{B(x,r)}\int_{ B(x,3r)} \left|\log\frac{r}{|y-\xi|}\right| \,d\omega^p(\xi)\,dm(y).
\end{align*}
Notice that for all $\xi\in B(x,3r)$,
$$\int_{B(x,3r)}\left|\log\frac{r}{|y-\xi|}\right| \,dm(y)\lesssim r^2.$$
So by Fubini we obtain
$$
\frac1{r^5}\iint_{B(x,r)\times B(x,r)} |B_{y,z}|\,dm(y)\,dm(z)\lesssim \frac{\omega^p(B(x,3r))}r\lesssim\M^1\omega^p(x)
.$$
Together with the bound for the term $A_{y,z}$, this gives
$$
\bigl|\wt\RR_r\omega^p(x)\bigr|\lesssim \frac{\omega^{p}(Q)}{\mu(Q)} +\M^1\omega^p(x) + \frac1{d(p)}\lesssim 1,
$$
since $\M^1\omega^p(x)\lesssim1$ by \eqref{eqcc0}.
\vv

It remains now to show \rf{eqlem33}. The argument uses the ideas in Lemma \ref{l:w>G} with some modifications. 
Recall that
\begin{align*}
A_{y,z} 
&= 
A_{y,z}(p) = 
\log\frac{|z-p|}{|y-p|} - \int_{\partial\Omega} \phi\left(\frac{\xi-x}{r}\right)\,\log\frac{|z-\xi|}{|y-\xi|} \,d\omega^p(\xi)
\\
&=:
\log\frac{|z-p|}{|y-p|}-v_{x,y,z}(p)
\end{align*}
where $y,z\in B(x,r)$ and $p$ is far away.
The two functions 
$$
q\longmapsto A_{y,z}(q)\qquad\text{ and }\qquad q\longmapsto \frac{c\,\omega^q(B(x,2\delta^{-1}r))}{\inf_{z\in B(x,2r)\cap \Omega} \omega_\Omega^{z}(B(x,2\delta^{-1}r))}$$
 are harmonic in $\Omega\setminus B(x,2r)$. Note that for all $q\in\partial B(x,2r)\cap\Omega$
we clearly have
$$|A_{y,z}(q)|\leq c\leq \frac{c\,\omega^q(B(x,2\delta^{-1}r))}{\inf_{z\in B(x,2r)\cap \Omega} \omega_\Omega^{z}(B(x,2\delta^{-1}r))}.$$
 Since $A_{y,z}(x)=0$ for all $x\in\pom\setminus B(x,3r)$ except for a polar set we can apply maximum principle and obtain \rf{eqlem33}, as desired (the maximum principle is justified in the same way as in the proof of Lemma \ref{l:w>G}).
\end{proof}

\vv
From the Key Lemma above we deduce the following.

\begin{lemma}\label{lemaxcor}
Let $Q\in\good$ be contained in some cube from the family $\wt{\DD}_0^{db}$, and $x\in Q$. Then we have
\begin{equation}\label{eqdk10}
\RR_{*,r(B_Q)}\omega^p(x) \leq  C(A,M,  T,  \tau,d(p)),
\end{equation}
where, to shorten notation, we wrote $d(p)= \dist(p,\partial\Omega)$.
\end{lemma}

\vv

\subsection{The end of the proof of Theorem \ref{teo1} (a)}\label{secend}

Set
$$G= F_M \cap \bigcup_{Q\in\wt{\DD}_0^{db}} Q  \setminus
 \bigcup_{Q\in\bad} Q.$$
and recall that, by
Lemma \ref{lemgg}, 
$$\omega^p(G)>0.$$
As shown in \rf{eqcc0}, we have
\begin{equation}\label{eqcc1}
\M^n\omega^p(x)\lesssim A\quad \mbox{ for $\omega^p$-a.e.\ $x\in G$.}
\end{equation} 
On the other hand, from Lemma \ref{lemaxcor} is also clear that
\begin{equation}\label{eqcc2}
\RR_{*}\omega^p(x)\leq C(A,M, T, \tau,d(p))\quad \mbox{ for $\omega^p$-a.e.\ $x\in G$.}
\end{equation}

Now we will apply the following result.

\begin{theorem}  \label{teo**}
Let $\sigma$ be a Radon measure with compact support on $\R^{n+1}$   and consider a $\sigma$-measurable set
$G$ with $\sigma(G)>0$ such that
$$G\subset\{x\in \R^{n+1}: 
\M^n\sigma(x) < \infty \mbox{ and } \,\RR_* \sigma(x) <\infty\}.$$
Then there exists a Borel subset $G_0\subset G$ with $\sigma(G_0)>0$  such
that $\sup_{x\in G_0}\M^n\sigma|_{G_0}(x)<\infty$ and  $\RR_{\sigma|_{G_0}}$ is bounded in $L^2(\sigma|_{G_0})$.
\end{theorem}

This result follows from the deep non-homogeneous Tb theorem of Nazarov, Treil and Volberg in \cite{NTV} (see also \cite{Volberg}) in combination with the methods in \cite{Tolsa-pams}. For the detailed proof in the case of the Cauchy
transform, see \cite[Theorem 8.13]{Tolsa-llibre}. The same arguments with very minor modifications work for the Riesz transform.
\vv

Recall that that $\partial\Omega$ is compact as we are in the case when $\Omega$ is bounded.  From \rf{eqcc1}, \rf{eqcc2} and Theorem \ref{teo**} applied to $\sigma=\omega^p$, we infer that there exists a subset $G_0\subset G$
such that the operator $\RR_{\omega^p|_{G_0}}$ is bounded in $L^2(\omega^p|_{G_0})$. By Theorem 1.1 of \cite{NToV-pubmat}
(or the David-L\'eger theorem  \cite{Leger} for $n=1$), we deduce
that $\omega^p|_{G_0}$ is $n$-rectifiable.


\vv
\section{Some corollaries}

Here we list a few corollaries of Theorem \ref{teo1}. \vv

From Theorem \ref{teo1} in combination with the results of \cite{Az} we obtain the following characterization of sets of absolute continuity for NTA domains. 

\begin{corollary}
Let $\Omega\subset \R^{n+1}$ be an NTA domain (or in the case $n=1$, any simply connected domain), $p\in \Omega$, $n\geq 1$, and let $E\subset \d\Omega$ be such that $\HH^{n}(E)<\infty$ and $\omega^{p}(E)>0$. Then $\omega^{p}|_{E}\ll \HH^{n}|_{E}$ if and only if $E$ may be covered by countably many $n$-dimensional Lipschitz graphs up to a set of $\omega^{p}$-measure zero.
\end{corollary}

The forward direction is just Theorem \ref{teo1} (a). The reverse direction for $n=1$ follows from either the Local F. and M. Riesz Theorem of Bishop and Jones (see Theorem 1 of \cite{BJ}) and in the $n>1$ case from the first main result from \cite{Az}, which says that, for an NTA domain, whenever a Lipschitz graph $\Gamma$ (amongst more general objects) has $\omega(\Gamma)>0$, then $\omega\ll\HH^{n}$ on $\Gamma$.

\begin{corollary}
Let $\Omega\subset \R^{n+1}$ be an NTA domain whose exterior $(\Omega^{c})^{\circ}$ is also NTA, $p\in \Omega$, $n\geq 1$, and let $E\subset \d\Omega$ be such that $0<\HH^{n}(E)<\infty$. Then $ \HH^{n}|_{E}\ll \omega^{p}|_{E}$ if and only if $E$ may be covered by countably many $n$-dimensional Lipschitz graphs up to a set of $\HH^{n}$-measure zero.
\end{corollary}

The first direction is Theorem \ref{teo1} (b), while the reverse direction follows from the second main result from \cite{Az}, which says that, for NTA domains with NTA complements, $\HH^{n}\ll \omega$ on any Lipschitz graph $\Gamma$.

\vvv

\end{document}